\documentclass{amse-new}
\UseRawInputEncoding
\numberwithin{equation}{section} 
\newtheorem{assumption}{Assumption}
\usepackage{color}
\begin{document}

 \PageNum{1}
 \Volume{201x}{Sep.}{x}{x}
 \OnlineTime{August 15, 202x}
 \DOI{0000000000000000}
 \EditorNote{Received x x, 202x, accepted x x, 201x}

\abovedisplayskip 6pt plus 2pt minus 2pt \belowdisplayskip 6pt
plus 2pt minus 2pt
\def\vsp{\vspace{1mm}}
\def\th#1{\vspace{1mm}\noindent{\bf #1}\quad}
\def\proof{\vspace{1mm}\noindent{\it Proof}\quad}
\def\no{\nonumber}
\newenvironment{prof}[1][Proof]{\noindent\textit{#1}\quad }
{\hfill $\Box$\vspace{0.7mm}}
\def\q{\quad} \def\qq{\qquad}
\allowdisplaybreaks[4]


\AuthorMark{Liu M. et al.}                             

\TitleMark{Trotter-Kato Approximations of Impulsive Neutral SPDEs}  

\title{Trotter-Kato Approximations of Impulsive Neutral SPDEs in Hilbert Spaces        
\footnote{Supported by the National Natural Science Foundation of China (Grant No. 12171361) and the Humanity and Social Science Youth foundation of Ministry of Education (Grant No. 20YJC790174).}}                  

\author{Ming \uppercase{Liu}}             
    {Address\,$:$School of Mathematical Sciences, TianGong University, No. 399 Binshui Road, Xiqing District, 300387 Tianjin, P.R.China\\
    E-mail\,$:$liuming@tiangong.edu.cn }

\author{Ling Fei \uppercase{Dai}}     
    {Address\,$:$School of Mathematical Sciences, TianGong University, No. 399 Binshui Road, Xiqing District, 300387 Tianjin, P.R.China\\
    E-mail\,$:$dlf970226@163.com  }

\author{Xia \uppercase{Zhang}\textsuperscript{1)}\thanks{1) Corresponding author}}     
    {Address\,$:$School of Mathematical Sciences, TianGong University, No. 399 Binshui Road, Xiqing District, 300387 Tianjin, P.R.China\\
    E-mail\,$:$zhangxia@tiangong.edu.cn  }

\maketitle%

\Abstract{This paper studies a class of impulsive neutral stochastic partial differential equations in real Hilbert spaces.
The main goal here is to consider the Trotter-Kato approximations of mild solutions of such equations in the $p$th-mean ($p\geq2$).
As an application, a classical limit theorem on the dependence of such equations on a parameter is obtained.
The novelty of this paper is that the combination of this approximating system and such equations has not been considered before.}      


\MRSubClass{}      

\section{Introduction}
The study on mild solutions of impulsive stochastic partial differential equations (SPDEs), which is a hot research theme
in the past two decades, had attracted a great interests of a lot of famous scholars, such as Ahmed\cite{91}, Anguraj and Vinodkumar\cite{AA09}, Da Prato and Zabczyk\cite{92a}, Govindan\cite{09,14}, Ichikawa\cite{82}, Mckibben\cite{11} and so on.
Many of these equations are often used to simulate the stochastic process that are found in the study of natural science, engineering and mathematical finance.

The paper aims to study the impulsive neutral SPDE in a real separable Hilbert space of the form:
\begin{align}
\label{system(1.1)}
 d[x(t)+f(t,\pi_tx)]&=[Ax(t)+a(t,\pi_tx)]dt+b(t,\pi_tx)d\omega(t),\ t>0,\ t\neq t_k,\\
 \label{system(1.2}
 \Delta x(t_k)&=x(t_k^+)-x(t_k^-)=I_k(x(t_k)),\ t=t_k,\ k=1,2,\dots,m,\\
\label{system(1.3)}
x(t)&=\varphi(t),\ t\in[-r,0]\ (0\leq r<\infty),
\end{align}
where $\pi_tx=\{x(t-r+s):0\leq s\leq r\}$, $A:\ D(A)\subseteq X\rightarrow X$ is the infinitesimal generator of a strongly continuous semigroup $\{S(t), t\geq0\}$ defined on $X$, $a:\ R^+\times X\rightarrow X(R^+=[0,\infty))$, $f:\ R^+\times X\rightarrow D(A^\alpha)$, $0<\alpha\leq1$ and $b:\ R^+\times X\rightarrow L(Y,X)$ are Borel-measurable. Here $\omega(t)$ is a $Y$-valued $Q$-Wiener process. The fixed moments of time $t_k$ satisfy $0<t_1<\dots<t_m<T,$ where $x(t_k^+)$ and $x(t_k^-)$ represent the right and left limits of $x(t)$ at $t=t_k,$ respectively. $\Delta x(t_k)=x(t_k^+)-x(t_k^-),$ represents the jump in the state $x$ at time $t_k$ with $I_k\in C(X,X)$ determining the size of the jump. The past stochastic process $\{\varphi(t), t\in[-r,0]\}$ has almost surely (a.s.) continuous sample paths with $E\|\varphi\|_c^p<\infty$, $p\geq2$.

Taniguchi\cite{98} investigated the existence and uniqueness of a mild solution of the stochastic evolution equations with finite delays.
Next, Govindan\cite{09} established the properties of mild solutions for SPDEs with neutral term.
Anguraj and Vinodkumar\cite{AA09} studied the existence, uniqueness and stability results of other form of impulsive stochastic semilinear neutral functional differential equations with infinite delays.
Chaudhary and Pandey\cite{19} considered the existence of  mild solutions of the impulsive neutral fractional stochastic integro-differential systems with state dependent delay. In addition, Deng, Shu and Mao\cite{18a} established a new impulsive-integral inequality to prove the exponential stability of mild solutions for a class of impulsive neutral SPDEs driven by fBm with noncompact semigroup. Baleanu, Annamalai, Kandasamy, et. al\cite{18b} studied the impulsive neutral SPDEs with Poisson jumps. Guo, Chen, Shu et. al\cite{20a} considered the Hyers-Ulam stability of the almost periodic solution to the fractional differential equation with impulse and fractional Brownian motion under nonlocal condition.
The classical Trotter-Kato approximating system had been well studied in Pazy\cite{83}.
Such approximations also had been considered in Govindan\cite{GTE20} and Kannan and Bharucha-Reid\cite{85} for the semilinear stochastic evolution equations.
Besides, one can refer to the work of Guo et. al\cite{GTX10,GTXYZY98,YZYGTX94}.
To the best of our knowledge, Trotter-Kato approximating system into impulsive neutral SPDEs has not been considered in the relevant literature.
Therefore, the purpose of this paper is to introduce Trotter-Kato approximations of equation (\ref{system(1.1)}) and prove the convergence of mild solutions of this approximating system in the $p$th-mean ($p\geq2$).
Finally, we give a classical limit theorem on the dependence of the equation (\ref{system(1.1)}) on a parameter as an application.

The remainder of this paper is organized as follows: in Section \ref{sec2}, we give some preliminaries; in Section \ref{sec3}, we consider the Trotter-Kato approximations and zeroth-order approximations results; in the last section, the classical limit theorem is given as an application.

\section{Preliminaries}\label{sec2}

Let $X$ and $Y$ be real separable Hilbert spaces and $L(Y,X)$ the space of bounded linear operators from $Y$ to $X$.
We use $|\cdot|$ to represent the norms of $X,$ $Y$ and $L(Y,X)$.
We write $L(X)$ for $L(X,X)$. Let $(\Omega,\mathcal{F},P,\{\mathcal{F}_t\}_{t\geq0})$ be a complete probability space with an increasing right continuous family $\left\{\mathcal{F}_t\right\}_{t\geq0}$ of complete sub-$\sigma$-algebras of $\mathcal{F}.$
Let $\beta_n(t), n=1,2,\dots,$ be a sequence of real-valued standard Brownian motions mutually independent defined on this probability space.
Let $\omega(t)=\sum_{n=1}^{\infty}\sqrt{\lambda_n}\beta_n(t)e_n,\ t\geq0,$ where $\lambda_n\geq0,\ n=1,2,\dots,$ are nonnegative real numbers and $\{e_n\},\ n=1,2,\dots,$ is a complete orthonormal basis in $Y.$
Let $Q\in L(Y)$ be an operator defined by $Qe_n=\lambda_ne_n.$
The underlying $Y$-valued stochastic process $\omega(t)$ is called a $Q$-Wiener process.
Let $h(t)$ be an $L(Y,X)$-valued function and $\lambda$ be a sequence $\{\sqrt{\lambda_1},\sqrt{\lambda_2},\dots\}.$ Then $|h(t)|_{\lambda}=\left\{\sum_{n=1}^{\infty}|\sqrt{\lambda_n}h(t)e_n|^2\right\}^{\frac{1}{2}}.$
If $|h(t)|_{\lambda}^2<\infty,$ then $h(t)$ is called $\lambda$-Hilbert-Schmidt operator.
Let $C:=C([-r,0];X)$ denote the space of continuous functions $\varphi:[-r,0]\rightarrow X$ endowed with the norm $\|\varphi\|_c=\sup_{-r\leq s\leq0}|\varphi(s)|.$
Let $C([-r,T],L^p(\Omega,X))$ be the space of continuous maps from $[-r,T]$ to $L^p(\Omega,X)$ satisfying the condition $\|x\|_T=\sup_{-r\leq s\leq T}|x(s)|,$ where $0<T<\infty$ and $\|\cdot\|_T$ denotes the norm of $C([-r,T],L^p(\Omega,X))$. For a continuous $\mathcal{F}_t$-adapted measurable $X$-valued stochastic process $x(t): \Omega\rightarrow X$, $t\geq-r$, we have the continuous $\mathcal{F}_t$-adapted measurable $X$-valued stochastic process $\pi_tx: \Omega\rightarrow X$, $t>0$, by setting $\pi_tx=x(t-r+s), 0\leq s\leq r$.

If $\{S(t), t\geq0\}$ is an analytic semigroup with infinitesimal generator A such that $0\in\rho(A)$ (the resolvent set of $A$), then we can define the fractional power $A^{\alpha}$ as a closed linear operator on its domain $D(A^{\alpha})$ for $0<\alpha\leq1.$ Furthermore, the subspace $D(A^{\alpha})$ is dense in $X$ and the expression $\|x\|_{\alpha}=|A^{\alpha}x|,\ x\in D(A^{\alpha}),$ which defines a norm on $D(A^{\alpha}).$
\begin{definition}
\rm(\cite{95})
\em Let $\Phi: [0,\infty)\rightarrow\sigma(\lambda)(Y,X)$ be a $\mathcal{F}_t$-adapted process. Then for any $\Phi$ satisfying $\int_0^tE|\Phi(s)|_\lambda^2ds<\infty$, we define the $X$-valued stochastic integral $\int_0^t\Phi(s)d\omega(s)\in X$ with respect to $\omega(t)$ by
$$\bigg(\int_0^t\Phi(s)d\omega(s),h\bigg)=\int_0^t\langle\Phi^*(s)h,d\omega(s)\rangle,\quad h\in X,$$
where $\Phi^*$ is the adjoint operator of $\Phi.$
\end{definition}
\begin{definition}
\label{def2.2}
\rm(\cite{95})
\em A stochastic process $\{x(t),t\in [-r,T]\}(0<T<\infty)$ is called a mild solution of Eq.(\ref{system(1.1)}) if
\begin{enumerate}[(a)]
\item$x(t)$ is $\mathcal{F}_t$-adapted with $\int_0^T|x(t)|^2dt<\infty,$ a.s.,
\item$x(t)$ satisfies the integral equation
\end{enumerate}
\begin{align*}
x(t)&=S(t)[\varphi(0)+f(0,\varphi)]-f(t,\pi_tx)-\int_0^tAS(t-s)f(s,\pi_sx)ds+\int_0^tS(t-s)a(s,\pi_sx)ds\\
&\quad+\int_0^tS(t-s)b(s,\pi_sx)d\omega(s)+\sum_{0<t_k<t}S(t-t_k)I_k(x(t_k)),\quad a.s.,\ t\in [0,T].
\end{align*}
\end{definition}
\begin{proposition}
\label{pro2.1}
\rm(\cite{92a})
\em Let $W_A^\Phi(t)=\int_0^tS(t-s)\Phi(s)dw(s),\ \forall t\in [0,T]$. Then there exists a constant $c(p,T)>0$ for any $p>2$ and $T\geq0$ such that the stochastic convolution $W_A^\Phi$ has a proper modification that
    $$E\sup_{0\leq t\leq T}|W_A^\Phi(t)|^p\leq c(p,T)\sup_{0\leq t\leq T}\|S(t)\|^pE\int_0^t|\Phi(s)|_\lambda^pds.$$
Moreover, if $E\int_0^T|\Phi(s)|_\lambda^pds<\infty$, then there exists a continuous version of the process $\{W_A^\Phi,t\geq0\}$.
\end{proposition}
\begin{proposition}
\label{pro2.2}
\rm(\cite{92a})
\em Suppose that $A$ generates a contraction semigroup. Then the process $W_A^\Phi(\cdot)$ has a continuous modification and there exists a constant $\kappa>0$ such that
    $$E\sup_{0\leq t\leq T}|W_A^\Phi(t)|^2\leq \kappa E\int_0^t|\Phi(s)|_\lambda^2ds.$$
\end{proposition}

\section{Trotter-Kato approximations}\label{sec3}

In this section, we first inject the idea of Trotter-Kato approximations into equation (\ref{system(1.1)}) to obtain equation (\ref{system(3.1)}), in which the mild solution of this equation also belongs to $C([0,T],L^p(\Omega,X)).$ Next, we prove the mild solution of equation (\ref{system(3.1)}) converging in the pth-mean $(p\geq2)$ to the mild solution of equation (\ref{system(1.1)}) in Theorem \ref{the(3.2)} according to the properties of the solution space and the function terms in the equations.
We impose the following assumption to study the main results.
\begin{assumption}
\label{A1}
\ \
\begin{enumerate}[(i)]
\item$A$ is the infinitesimal generator of a analytic semigroup of bounded linear operators $\{S(t),t\geq0\}$
in $X.$
\item For $p \geq 2$, the functions $a(t,\cdot)$ and $b(t,\cdot)$ satisfy the Lipschitz and linear growth conditions:
$$|a(t,\pi_tx)-a(t,\pi_ty)|^p \leq C_{1}|x-y|^p,\ C_1>0,$$
$$|b(t,\pi_tx)-b(t,\pi_ty)|_\lambda^p \leq C_2|x-y|^p,\ C_2>0,$$
$$|a(t,\pi_tx)|^p+|b(t,\pi_ty)|_\lambda^p\leq C_3(1+|x|^p),\ C_3>0$$
for any $x,\ y\in X.$
\item$f(t,\cdot)$ is a continuous function and satisfies:
$$|A^\alpha f(t,\pi_tx)-A^\alpha f(t,\pi_ty)|\leq C_{4}|x-y|,\ C_4>0,$$
$$|A^\alpha f(t,\pi_tx)|\leq C_{5}(1+|x|),\ C_5>0$$
for any $x,\ y\in X$.
\item$I_k$ is a continuous function and satisfies:
$$|I_k(x(t_k))-I_k(y(t_k))|\leq h_k|x-y|,\ h_k>0,$$
$$I_k(0)\leq h_0$$
for any $x,\ y\in X$.
\end{enumerate}
\end{assumption}
\begin{lemma}
\label{lem(3.1)}
\em Let Assumption (\ref{A1}) hold. Suppose that the semigroup $\{S(t),t\geq0\}$ is a contraction semigroup for the case $p=2$. Then there exists a unique mild solution $x$ in $C([0,T],L^p(\Omega,X))$ of system (\ref{system(1.1)}) provided $L\|A^{-\alpha}\|+Me^{\delta T}\sum_{k=1}^mh_k<1$ for any $p\geq2,$ where $L=\max\{C_4,C_5\}$ and $1/p<\alpha<1.$
\end{lemma}
\begin{proof}
By Anguraj and Vinodkumar \cite{AA09} and Govindan\cite{09}, this lemma can be proved according to Picard's iterations and Borel-Cantelli Lemma.
\end{proof}
Next, we study the family of stochastic neutral partial differential equations
\begin{align}
\label{system(3.1)}
d[x_n(t)+f(t,\pi_tx_n)]&=[A_nx_n(t)+a(t,\pi_tx_n)]dt+b(t,\pi_tx_n)d\omega(t),\ t\in[0,T],\ t\neq t_k,\\
\label{system(3.2)}
 \Delta x_n(t_k)&=x_n(t_k^+)-x_n(t_k^-)=I_k(x_n(t_k)),\ t=t_k,\ k=1,2,\dots,m,\\
\label{system(3.3)}
x_n(t)&=\varphi(t),\ t\in[-r,0],
\end{align}
where $A_n, n=1,2,3,\dots,$ are the infinitesimal generators of strongly continuous analytical semigroups $\{S_n(t),t\geq0\}$ of bounded linear operators on $X,$ respectively.

For each $n=1,2,3,\dots,$ by Lemma 3.1, System (\ref{system(3.1)})-(\ref{system(3.3)}) has a unique mild solution $x_n\in C([0,T],L^p(\Omega,X)).$ Hence, $x_n(t)$ satisfies the stochastic integral equation
\begin{align*}
x_n(t)&=S_n(t)[\varphi(0)+f(0,\varphi)]-f(t,\pi_tx_n)-\int_0^tA_nS_n(t-s)f(s,\pi_sx_n)ds\\
&\quad+\int_0^tS_n(t-s)a(s,\pi_sx_n)ds+\int_0^tS_n(t-s)b(s,\pi_sx_n)d\omega(s)\\
&\quad+\sum_{0<t_k<t}S_n(t-t_k)I_k(x_n(t_k)),\quad a.s.,\ t\in [0,T].
\end{align*}

\begin{lemma}
\label{lem3.2}
Let the operators $A$ and $A_n,n=1,2,3,\dots,$ are densely defined, closed and uniformly sectorial on $X,$ then the operators $A$ and $A_n$ generate strongly continuous analytic semigroups $S(t)$ and $S_n(t)$, respectively, which satisfy the uniform bounds
$$\|S(t)\|,\ \|S_n(t)\|\leq Me^{\delta t},\ \forall t\geq0$$ and
$$\|AS(t)\|,\ \|A_nS_n(t)\|\leq\frac{M'}{t}e^{\delta t},\ \forall t>0,$$
where $M\geq1,$ $M'\geq1$ and $\delta\in \mathbb{R}.$
\end{lemma}
\begin{proof}
Recall that a closed operator $A$ generates a strongly continuous analytic semigroup on a Banach space $X$ if and only if $A$ is densely defined and sectorial, that is, there exist constants $M\geq1$ and $\delta\in \mathbb{R}$ such that $\{\lambda\in C: Re\lambda>\delta\}$ is contained in the resolvent set $\rho(A)$ and
$$\mathop{{\rm sup}}_{Re\lambda>\delta}\|(\lambda-\delta)R(\lambda,A)\|\leq M,$$
the constants $M$ and $\delta$ are called the sectoriality constants of $A.$  in this context, we say that $A$ is sectorial of type $(M,\delta).$ Thus, $A$ generates a strongly continuous analytic semigroup $\{S(t), t\geq0\}$ on $X,$ it follows that $S(t)$ maps $X$ into the domain $D(A)$ and $\lim\sup_{t\downarrow0}t\|AS(t)\|<\infty.$ Since $A_n,n=1,2,3,\dots,$ satisfy the same conditions, we know that $A_n$ generates a strongly continuous analytic semigroup $\{S_n(t), t\geq0\}$ on $X$ and $\lim\sup_{t\downarrow0}t\|A_nS_n(t)\|<\infty$ for each $n=1,2,3,\dots.$ Therefore, the uniform bounds of the operators and semigroups have been proved.
\end{proof}
Now we make the following assumption:
\begin{assumption}
\label{A2}
\
\begin{enumerate}[(i)]
\item The operators $A$ and $A_n,n=1,2,3,\dots,$ are densely defined, closed and uniformly
sectorial on $X$ in the sense, there exist $M\geq1$ and $\delta\in \mathbb{R}$ such that $A$ and $A_n$ are sectorial of type $(M,\delta)$ for each $n=1,2,3,\dots.$
\item The operators $A_n$ converge to $A$ in the strong resolvent sense:
$$\lim_{n\rightarrow\infty}R(\lambda,A_n)x=R(\lambda,A)x$$
for any $Re\lambda>\delta$ and $x\in X.$
\end{enumerate}
\end{assumption}
\begin{remark}
If Assumption (\ref{A2})(i) holds, then according to Lemma \ref{lem3.2}, $A$ and $A_n,n=1,2,3,\dots,$ generate strongly continuous analytic semigroups $S(t)$ and $S_n(t),$ respectively, which satisfy the uniform bounds in Lemma \ref{lem3.2}.
\end{remark}

The following Trotter-Kato approximation theorem is well known.
\begin{proposition}
\label{Pro2.3}
\rm{(\cite{11'})}
\em Let Assumption (\ref{A2}) hold, then we have
$$\lim_{n\rightarrow \infty}S_n(t)x=S(t)x$$
for any $t\in[0,\infty)$ and $x\in X,$ where the convergence is uniform on compact subsets of $[0,\infty)\times X$ and
$$\lim_{n\rightarrow \infty}A_nS_n(t)x=AS(t)x$$
for any $t\in(0,\infty)$ and $x\in X,$ where the convergence is uniform on compact subsets of $(0,\infty)\times X.$
\end{proposition}
\begin{theorem}
\label{the(3.2)}
Suppose that Assumptions (\ref{A1}) and (\ref{A2}) hold. Let $x(t)$ and $x_n(t)$ be the mild solutions of equations (\ref{system(1.1)}) and (\ref{system(3.1)}), respectively. Then, for each $T\in(0,\infty),$
$$\sup_{0\leq t\leq T}E|x_n(t)-x(t)|^p=0\ \ as\ n\rightarrow\infty.$$
\end{theorem}
\begin{proof}
First, we claim that there exists a positive constant $C(T)>0$ such that the mild solution of equation (\ref{system(1.1)}) satisfies
$$\mathop{{\rm sup}}_{0\leq t\leq T}E|x(t)|^p\leq C(T)$$
for each $0<T<\infty.$

Now we consider the mild solution of equation (\ref{system(1.1)}).
\begin{align*}
|x(t)|&\leq Me^{\delta T}|\varphi(0)|+Me^{\delta T}\|A^{-\alpha}\|C_5(1+\|\varphi\|_c)+\|A^{-\alpha}\|C_5(1+|x(t)|)\\
&\quad+\left|\int_0^tAS(t-s)f(s,\pi_sx)ds\right|+\left|\int_0^tS(t-s)a(s,\pi_sx)ds\right|\\
&\quad+\left|\int_0^tS(t-s)b(s,\pi_sx)d\omega(s)\right|+Me^{\delta T}\sum_{k=1}^m|I_k(x(t_k))|\quad a.s.\ t\in[0,T].
\end{align*}
Therefore, by Proposition \ref{pro2.1} (or Proposition \ref{pro2.2} for $p=2$) and Lemma \ref{lem3.2}  we obtain
\begin{align*}
&[1-C_5\|A^{-\alpha}\|-Me^{\delta T}\sum_{k=1}^mh_k]^p\mathop{{\rm sup}}_{0\leq t\leq T}E|x(t)|^p\\
&\leq4^{p-1}\bigg\{E[Me^{\delta t}|\varphi(0)|+Me^{\delta t}\|A^{-\alpha}\|C_5(1+\|\varphi\|_c)+C_5\|A^{-\alpha}\|+Me^{\delta T}mh_0]^p\\
&\quad +2^{p-1}(T\frac{M'e^{\delta T}}{t}C_5\|A^{-\alpha}\|)^p+T^pM^pe^{p\delta T}C_3+TM^pe^{p\delta T}c(p,T)C_3\\
&\quad +[{(2T)}^{p-1}(\frac{M'e^{\delta T}}{t}C_5\|A^{-\alpha}\|)^p+T^{p-1}M^pe^{p\delta T}C_3+M^pe^{p\delta T}c(p,T)C_3]\int_0^t\mathop{{\rm sup}}_{0\leq s\leq t}E|x(s)|^pds\bigg\}.
\end{align*}
An application of the Bellman-Gronwall's Lemma yields
\begin{align*}
\mathop{{\rm sup}}_{0\leq t\leq T}E|x(t)|^p&\leq\frac{C_1(T)}{[1-C_5\|A^{-\alpha}\|-Me^{\delta T}\sum_{k=1}^mh_k]^p}\exp\left\{\frac{tC_2(T)}{[1-C_5\|A^{-\alpha}\|-Me^{\delta T}\sum_{k=1}^mh_k]^p}\right\}\\
&<C(T),
\end{align*}
where $C_1(T),$ $C_2(T)$ and $C(T)$ are positive constants and the claim follows.

Next, considering $x_n(t)-x(t)$, we have
\begin{align}
\label{eq(3.4)}
&[1-C_4\|A^{-\alpha}\|-Me^{\delta T}\sum_{k=1}^mh_k]^p\mathop{{\rm sup}}_{0\leq t\leq T}E|x_n(t)-x(t)|^p\nonumber\\
&\leq9^{p-1}\bigg\{E|[S_n(t)-S(t)]\varphi(0)|^p+E|[S_n(t)-S(t)]f(0,\varphi)|^p\nonumber\\
&\quad+E\left|\int_0^t[AS(t-s)-A_nS_n(t-s)]f(s,\pi_sx)ds\right|^p+E\left|\int_0^t[S_n(t-s)-S(t-s)]a(s,\pi_sx)ds\right|^p\nonumber\\
&\quad+E\left|\int_0^t[S_n(t-s)-S(t-s)]b(s,\pi_sx)d\omega(s)\right|^p+E\left|\int_0^tA_nS_n(t-s)[f(s,\pi_sx)-f(s,\pi_sx_n)]ds\right|^p\nonumber\\
&\quad+E\left|\int_0^tS_n(t-s)[a(s,\pi_sx_n)-a(s,\pi_sx)]ds\right|^p+E\left|\int_0^tS_n(t-s)[b(s,\pi_sx_n)-b(s,\pi_sx)]d\omega(s)\right|^p\nonumber\\
&\quad+E\left|\sum_{k=1}^m[S_n(t-t_k)-S(t-t_k)]I_k(x(t_k))\right|^p\bigg\}.
\end{align}

We estimate each term in the right hand side of (\ref{eq(3.4)}).
Since $E|[S_n(t)-S(t)]\varphi(0)|\leq2Me^{\delta T}E|\varphi(0)|$ and $E|[S_n(t)-S(t)]f(0,\varphi)|\leq2Me^{\delta T}C_5\|A^{-\alpha}\|(1+E\|\varphi\|_c)$
From Proposition \ref{Pro2.3} and the Lebesgue dominated convergence theorem, one can get
\begin{align}
\label{eq(3.5)}
E|[S_n(t)-S(t)]\varphi(0)|^p\rightarrow0\quad as\ n\rightarrow\infty
\end{align}
and
\begin{align}
\label{eq(3.6)}
E|[S_n(t)-S(t)]f(0,\varphi)|^p\rightarrow0\quad as\ n\rightarrow\infty.
\end{align}

Next, by Lemma \ref{lem3.2}, we have
\begin{align}
\label{eq(3.7)}
&E\left|\int_0^tA_nS_n(t-s)[f(s,\pi_sx)-f(s,\pi_sx_n)]ds\right|^p\nonumber\\
&\leq T^{p-1}E\int_0^t\|A_nS_n(t-s)\|^p|f(s,\pi_sx)-f(s,\pi_sx_n)|^pds\nonumber\\
&\leq T^{p-1}\left(\frac{M'e^{\delta T}}{t}C_4\|A^{-\alpha}\|\right)^p\int_0^tE|x_n(s)-x(s)|^pds.
\end{align}
By Assumption (\ref{A1}), it follows that
\begin{align}
\label{eq(3.8)}
E\left|\int_0^tS_n(t-s)[a(s,\pi_sx_n)-a(s,\pi_sx)]ds\right|^p&\leq T^{p-1}M^pe^{p\delta T}C_1\int_0^tE|x_n(s)-x(s)|^pds.
\end{align}
Using Proposition \ref{pro2.1}, we estimate the stochastic integral as follows:
\begin{align}
\label{eq(3.9)}
E\left|\int_0^tS_n(t-s)[b(s,\pi_sx_n)-b(s,\pi_sx)]d\omega(s)\right|^p\leq c(p,T)M^pe^{p\delta T}C_2\int_0^tE|x_n(s)-x(s)|^pds.
\end{align}
Using the estimates (\ref{eq(3.7)})-(\ref{eq(3.8)}) and the inequality (\ref{eq(3.4)}) reduces to
\begin{align*}
&[1-C_4\|A^{-\alpha}\|-Me^{\delta T}\sum_{k=1}^mh_k]^p\mathop{{\rm sup}}_{0\leq t\leq T}E|x_n(t)-x(t)|^p\\
&\leq\beta(n,T)+9^{p-1}\bigg[T^{p-1}\left(\frac{M'e^{\delta T}}{t}C_4\|A^{-\alpha}\|\right)^p+T^{p-1}M^pe^{p\delta T}C_1\nonumber\\
&\quad+c(p,T)M^pe^{p\delta T}C_2\bigg]\int_0^tE|x_n(s)-x(s)|^pds,
\end{align*}
where
\begin{align}
\label{eq(3.10)}
\beta(n,T)&=9^{p-1}\bigg\{E|[S_n(t)-S(t)]\varphi(0)|^p+E|[S_n(t)-S(t)]f(0,\varphi)|^p\nonumber\\
&\quad+E\left|\int_0^t[AS(t-s)-A_nS_n(t-s)]f(s,\pi_sx)ds\right|^p\nonumber\\
&\quad+E\left|\int_0^t[S_n(t-s)-S(t-s)]a(s,\pi_sx)ds\right|^p\nonumber\\
&\quad+E\left|\int_0^t[S_n(t-s)-S(t-s)]b(s,\pi_sx)d\omega(s)\right|^p\nonumber\\
&\quad+E\left|\sum_{k=1}^m[S_n(t-t_k)-S(t-t_k)]I_k(x(t_k))\right|^p\bigg\}.
\end{align}
The first two terms in the right hand side of (\ref{eq(3.10)}) tends to zero as $n\rightarrow\infty$ by (\ref{eq(3.5)})-(\ref{eq(3.6)}). By Proposition \ref{Pro2.3}, we now have
\begin{align*}
&E\left|\int_0^t[AS(t-s)-A_nS_n(t-s)]f(s,\pi_sx)ds\right|^p\\
&\leq T^{p-1}E\int_0^t\|AS(t-s)-A_nS_n(t-s)\|^p|f(s,\pi_sx)|^pds\\
&\leq2^{2p-1}\left(T\frac{M'e^{\delta T}}{t}C_5\|A^{-\alpha}\|\right)^p[1+C(T)]<\infty.
\end{align*}
Hence, the third term of (\ref{eq(3.10)}) tends to zero as $n\rightarrow\infty$ in view of Proposition \ref{Pro2.3} together with the dominated convergence theorem. Regarding the fourth term, note that
\begin{align*}
E\left|\int_0^t[S_n(t-s)-S(t-s)]a(s,\pi_sx)ds\right|^p
&\leq(2TMe^{\delta T})^pC_3[1+C(T)]<\infty.
\end{align*}
By the same theorem, this term also tends to zero. Next, considering the stochastic integral term,
\begin{align*}
&E\left|\int_0^t[S_n(t-s)-S(t-s)]b(s,\pi_sx)d\omega(s)\right|^p\\
&\leq c(p,T)E\int_0^t\|S_n(t-s)-S(t-s)\|^p|b(s,\pi_sx)|_\lambda^pds\\
&\leq c(p,T)(2Me^{\delta T})^pTC_3(1+C(T))<\infty.
\end{align*}
Therefore, the stochastic integral term also tends to zero. Finally, we study the last term,
\begin{align*}
&E\left|\sum_{k=1}^m[S_n(t-t_k)-S(t-t_k)]I_k(x(t_k))\right|^p\\
&\leq m^{p-1}\sum_{k=1}^m[\|S_n(t-t_k)-S(t-t_k)\|^p|I_k(x(t_k))|^p]\\
&\leq2^{2p-1}m^{p-1}M^pe^{p\delta T}\sum_{k=1}^m[|I_k(x(t_k))-I_k(0)|^p+|I_k(0)|^p]\\
&\leq2^{2p-1}m^{p-1}M^pe^{p\delta T}\sum_{k=1}^m{h_k}^{p}C(T)+m{h_0}^p<\infty.
\end{align*}
By Lebesgue¡¯s dominated convergence theorem,  it can be shown as before that the last term also tends to zero.
Thus, $\beta(n,T)\rightarrow0$ as $n\rightarrow\infty,$ which completes the proof.
\end{proof}
Let us next study the zeroth-order approximations, that is approximating a impulsive neutral SPDE by a impulsive neutral deterministic PDE.

We consider the neutral SPDE:
\begin{align}
\label{system(3.11)}
d[x_\varepsilon(t)+f(t,\pi_tx_\varepsilon)]&=[A_\varepsilon x_\varepsilon(t)+a(t,\pi_tx_\varepsilon)]dt+\varepsilon b(t,\pi_tx_\varepsilon)d\omega(t),\ t\in[0,T],\ t\neq t_k,\\
\label{system(3.12}
 \Delta x_\varepsilon(t_k)&=x_\varepsilon(t_k^+)-x_\varepsilon(t_k^-)=I_k(x_\varepsilon(t_k)),\ t=t_k,\ k=1,2,\dots,m,\\
\label{system(3.13)}
x_\varepsilon(t)&=\varphi(t),\ t\in[-r,0],
\end{align}
where $A_\varepsilon(\varepsilon>0)$ is the infinitesimal generator of a strongly continuous analytic semigroup $\{S_\varepsilon(t),t\geq0\}$ of bounded linear operators on $X,$ along with the impulsive neutral deterministic PDE:
\begin{align}
\label{system(3.14)}
d[\overline{x}(t)+f(t,\pi_t\overline{x})]&=[A\overline{x}(t)+a(t,\pi_t\overline{x})]dt,\ t\in[0,T],\ t\neq t_k,\\
\label{system(3.15}
 \Delta \overline{x}(t_k)&=\overline{x}(t_k^+)-\overline{x}(t_k^-)=I_k(\overline{x}(t_k)),\ t=t_k,\ k=1,2,\dots,m,\\
\label{system(3.16)}
\overline{x}(t)&=\varphi(t),\ t\in[-r,0],
\end{align}
The mild solutions of equations (\ref{system(3.11)}) and (\ref{system(3.14)}) are
\begin{align}
\label{eq(3.17)}
x_\varepsilon(t)&=S_\varepsilon(t)[\varphi(0)+f(0,\varphi)]-f(t,\pi_tx_\varepsilon)-\int_0^tA_\varepsilon S_\varepsilon(t-s)f(s,\pi_sx_\varepsilon)ds\nonumber\\
&\quad+\int_0^tS_\varepsilon(t-s)a(s,\pi_sx_\varepsilon)ds+\int_0^tS_\varepsilon(t-s)b(s,\pi_sx_\varepsilon)d\omega(s)\nonumber\\
&\quad+\sum_{0<t_k<t}S_\varepsilon(t-t_k)I_k(x_\varepsilon(t_k)),\quad a.s.,\ t\in [0,T],
\end{align}
and
\begin{align}
\label{eq(3.18)}
\overline{x}(t)&=S(t)[\varphi(0)+f(0,\varphi)]-f(t,\pi_t\overline{x})-\int_0^tAS(t-s)f(s,\pi_s\overline{x})ds\nonumber\\
&\quad+\int_0^tS(t-s)a(s,\pi_s\overline{x})ds+\sum_{0<t_k<t}S(t-t_k)I_k(\overline{x}(t_k)),\quad a.s.,\ t\in [0,T],
\end{align}
respectively. For each $\varepsilon>0,$ one can show by Lemma \ref{lem(3.1)} that equation (\ref{system(3.11)}) has a unique mild solution $x_\varepsilon\in C([0,T],L^p(\Omega,X)),$ given by equation (\ref{eq(3.17)}); and equation (\ref{system(3.14)}) also has a unique mild solution given by equation (\ref{eq(3.18)}) when $b\equiv0$ as a special case.

We now make the following assumption to consider the next result.
\begin{assumption}
\label{A3}
\
\begin{enumerate}[(i)]
\item Let $A_\varepsilon(\varepsilon>0)$ is densely defined, closed and uniformly
sectorial on $X$ in the sense, there exist $M\geq1$ and $\delta\in \mathbb{R}$ such that $A_\varepsilon$ is sectorial of type $(M,\delta)$ for each $\varepsilon>0.$
\item The operators $A_\varepsilon$ converge to $A$ in the strong resolvent sense:
$$\lim_{\varepsilon\downarrow0}R(\lambda,A_\varepsilon)x=R(\lambda,A)x$$
for any $Re\lambda>\delta$ and $x\in X.$
\end{enumerate}
\end{assumption}
\begin{remark}
Under Assumption (A3), we have $\lim_{\varepsilon\downarrow0}S_\varepsilon(t)x=S(t)x,$ uniformly on compact subsets of $[0,\infty)\times X$ and $\lim_{\varepsilon\downarrow0}A_\varepsilon S_\varepsilon(t)x=AS(t)x,$ uniformly on compact subsets of $(0,\infty)\times X.$
\end{remark}
In the following result, we estimate the error in the approximations.
\begin{theorem}
\label{the3.7}
Suppose that Assumptions (A1) and (A3) hold. Let $x_\varepsilon(t)$ and $\overline{x}(t)$ be the mild solutions given by \ref{eq(3.17)} and \ref{eq(3.18)}, respectively. Then
$$E|x_\varepsilon(t)-\overline{x}(t)|^p\leq\tau(\varepsilon)\phi(t),$$
where $\phi(t)$ is a positive exponentially increasing function and $\tau(\varepsilon)$ is a positive function decreasing monotonically to zero as $\varepsilon\downarrow0.$
\end{theorem}
\begin{proof}
We study
\begin{align*}
x_\varepsilon(t)-\overline{x}(t)&=[S_\varepsilon(t)-S(t)][\varphi(0)+f(0,\varphi)]+[f(t,\pi_t\overline{x})-f(t,\pi_tx_\varepsilon)]\nonumber\\
&\quad+\int_0^tA_\varepsilon S_\varepsilon(t-s)[f(s,\pi_s\overline{x})-f(s,\pi_sx_\varepsilon)]ds\nonumber\\
&\quad+\int_0^t[AS(t-s)-A_\varepsilon S_\varepsilon(t-s)]f(s,\pi_s\overline{x})ds\nonumber\\
&\quad+\int_0^tS_\varepsilon(t-s)[a(s,\pi_sx_\varepsilon)-a(s,\pi_s\overline{x})]ds+\int_0^t[S_\varepsilon(t-s)-S(t-s)]a(s,\pi_s\overline{x})ds\nonumber\\
&\quad+\varepsilon\int_0^tS_\varepsilon(t-s)b(s,\pi_sx_\varepsilon)d\omega(s)+\sum_{k=1}^mS_\varepsilon(t-t_k)[I_k(x_\varepsilon(t_k))-I_k(\overline{x}(t_k))]\nonumber\\
&\quad+\sum_{k=1}^m[S_\varepsilon(t-t_k)-S(t-t_k)]I_k(\overline{x}(t_k)),\ P-a.s.
\end{align*}
for any $t\in[0,T].$ By Assumption (A1), we have
\begin{align}
\label{eq(3.19)}
&[1-C_4\|A^{-\alpha}\|-Me^{\delta T}\sum_{k=1}^mh_k]|x_\varepsilon(t)-\overline{x}(t)|\nonumber\\
&\leq\left|[S_\varepsilon(t)-S(t)][\varphi(0)+f(0,\varphi)]\right|+\left|\int_0^tA_\varepsilon S_\varepsilon(t-s)[f(s,\pi_s\overline{x})-f(s,\pi_sx_\varepsilon)]ds\right|\nonumber\\
&\quad+\left|\int_0^t[AS(t-s)-A_\varepsilon S_\varepsilon(t-s)]f(s,\pi_s\overline{x})ds\right|+\left|\int_0^tS_\varepsilon(t-s)[a(s,\pi_sx_\varepsilon)-a(s,\pi_s\overline{x})]ds\right|\nonumber\\
&\quad+\left|\int_0^t[S_\varepsilon(t-s)-S(t-s)]a(s,\pi_s\overline{x})ds\right|+\left|\varepsilon\int_0^tS_\varepsilon(t-s)b(s,\pi_sx_\varepsilon)d\omega(s)\right|\nonumber\\
&\quad+\left|\sum_{k=1}^m[S_\varepsilon(t-t_k)-S(t-t_k)]I_k(\overline{x}(t_k))\right|,\ P-a.s.
\end{align}
for any $t\in[0,T].$

We now estimate each term on the right hand side of equation (\ref{eq(3.19)}). Since $S_\varepsilon(t)x\rightarrow S(t)x$ as $\varepsilon\downarrow0,$ uniformly on compact subsets of $[0,\infty)\times X,$ then there exists an $\varepsilon_1>0$ and $k_1>0$ such that $E\left|[S_\varepsilon(t)-S(t)][\varphi(0)+f(0,\varphi)]\right|^p\leq k_1a_1(\varepsilon)$ for any $t\in[0,T],$ where $0<a_1(\varepsilon)\downarrow0$ as $\varepsilon_1>\varepsilon\downarrow0.$

Next, by Lemma \ref{lem3.2}, we get
\begin{align*}
&E\left|\int_0^tA_\varepsilon S_\varepsilon(t-s)[f(s,\pi_s\overline{x})-f(s,\pi_sx_\varepsilon)]ds\right|^p\\
&\leq T^{p-1}\frac{{M'}^pe^{p\delta T}}{t}E\int_0^t\left|f(s,\pi_s\overline{x})-f(s,\pi_sx_\varepsilon)\right|^pds\\
&\leq T^{p-1}\left(\frac{{M'}^pe^{p\delta T}}{t}C_4\|A^{-\alpha}\|\right)^p\int_0^tE|x_\varepsilon(s)-\overline{x}(s)|^pds
\end{align*}
and
\begin{align*}
E\left|\int_0^tS_\varepsilon(t-s)[a(s,\pi_sx_\varepsilon)-a(s,\pi_s\overline{x})]ds\right|^p&\leq T^{p-1}M^pe^{p\delta T}C_1\int_0^tE|x_\varepsilon(s)-\overline{x}(s)|^pds.
\end{align*}
Regarding the third term, note that
\begin{align*}
E\left|\int_0^t[AS(t-s)-A_\varepsilon S_\varepsilon(t-s)]f(s,\pi_s\overline{x})ds\right|^p&\leq2^{2p-1}\left(T\frac{M'e^{\delta T}}{t}C_5\|A^{-\alpha}\|\right)^p[1+C(T)]<\infty.
\end{align*}
Thus, there exists an $\varepsilon_2>0$ and $k_2>0$ such that
$$E\left|\int_0^t[AS(t-s)-A_\varepsilon S_\varepsilon(t-s)]f(s,\pi_s\overline{x})ds\right|^p\leq k_2a_2(\varepsilon),$$
uniformly for any $t\in[0,T],$ where $0<a_2(\varepsilon)\downarrow0$ as $\varepsilon_2>\varepsilon\downarrow0.$ Now we consider the fifth and seventh terms of (\ref{eq(3.19)}):
\begin{align*}
E\left|\int_0^t[S_\varepsilon(t-s)-S(t-s)]a(s,\pi_s\overline{x})ds\right|^p
&\leq(2TMe^{\delta T})^pC_3[1+C(T)]<\infty
\end{align*}
and
\begin{align*}
E\left|\sum_{k=1}^m[S_\varepsilon(t-t_k)-S(t-t_k)]I_k(\overline{x}(t_k))\right|^p&\leq2^{2p-1}m^{p-1}M^pe^{p\delta T}\sum_{k=1}^m{h_k}^{p}C(T)+m{h_0}^p<\infty.
\end{align*}
Hence, there exists an $\varepsilon_3>0,$ $k_3>0$ and an $\varepsilon_4>0,$ $k_4>0$ such that
$$E\left|\int_0^t[S_\varepsilon(t-s)-S(t-s)]a(s,\pi_s\overline{x})ds\right|^p\leq k_3a_3(\varepsilon)$$
and
$$E\left|\sum_{k=1}^m[S_\varepsilon(t-t_k)-S(t-t_k)]I_k(\overline{x}(t_k))\right|^p\leq k_4a_4(\varepsilon),$$
respectively, uniformly for any $t\in[0,T],$ where $0<a_3(\varepsilon)\downarrow0$ as $\varepsilon_3>\varepsilon\downarrow0$ and $0<a_4(\varepsilon)\downarrow0$ as $\varepsilon_4>\varepsilon\downarrow0.$
Finally, we study the stochastic integral term:
\begin{align*}
\left|\varepsilon\int_0^tS_\varepsilon(t-s)b(s,\pi_sx_\varepsilon)d\omega(s)\right|&\leq \left|\varepsilon\int_0^tS_\varepsilon(t-s)[b(s,\pi_sx_\varepsilon)-b(s,\pi_s\overline{x})]d\omega(s)\right|\\
&\quad+\left|\varepsilon\int_0^tS_\varepsilon(t-s)b(s,\pi_s\overline{x})d\omega(s)\right|\\
&\triangleq J_1+J_2.
\end{align*}
Using Proposition \ref{pro2.1} (or Proposition \ref{pro2.2} for $p=2$), we have
\begin{align*}
E{J_1}^p&\leq\varepsilon c(p,T)M^pe^{p\delta T}C_2\int_0^tE|x_\varepsilon(s)-\overline{x}(s)|^pds
\end{align*}
and
\begin{align*}
E{J_2}^p&\leq\varepsilon c(p,T)M^pe^{p\delta T}C_3T[1+C(T)]<\infty.
\end{align*}
Thus, there exists an an $\varepsilon_5>0$ and $k_5>0$ such that $E{J_2}^p\leq k_5a_5(\varepsilon),$ where $0<a_5(\varepsilon)\downarrow0$ as $\varepsilon_5>\varepsilon\downarrow0.$

Set $\tau(\varepsilon)=8^{p-1}[k_1a_1(\varepsilon)+k_2a_2(\varepsilon)+k_3a_3(\varepsilon)+k_4a_4(\varepsilon)+k_5a_5(\varepsilon)]$ for $0<\varepsilon<\varepsilon_0,$ where $\varepsilon_0<\min\{\varepsilon_i, i=1,2,\dots,5\}.$ Consequently, for $0<\varepsilon<\varepsilon_0,$ we have
\begin{align*}
E|x_\varepsilon(t)-\overline{x}(t)|^p&\leq\tau(\varepsilon)+8^{p-1}e^{p\delta T}\left[T^{p-1}\left(\frac{M'}{t}C_4\|A^{-\alpha}\|\right)^p+T^{p-1}M^pC_1+\varepsilon c(p,T)M^pC_2\right]\\
&\quad\times\int_0^tE|x_\varepsilon(s)-\overline{x}(s)|^pds.
\end{align*}
Invoking Bellman-Gronwall's lemma, one obtains
$$E|x_\varepsilon(t)-\overline{x}(t)|^p\leq\tau(\varepsilon)\phi(t),\ \forall t\in[0,T],$$
where $\phi(t)=\exp\left\{8^{p-1}e^{p\delta T}\left[T^{p-1}\left(\frac{M'}{t}C_4\|A^{-\alpha}\|\right)^p+T^{p-1}M^pC_1+\varepsilon c(p,T)M^pC_2\right]t\right\}.$
\end{proof}

\section{Application}\label{sec4}

In this section, as an application of the results in Section \ref{sec3}, we consider a classical limit theorem on the dependence of System (\ref{system(1.1)})-(\ref{system(1.3)}) on a parameter. Let us consider the family of impulsive neutral SPDEs:
\begin{align}
\label{system(4.1)}
 d[x_n(t)+f_n(t,\pi_tx_n)]&=[A_nx_n(t)+a_n(t,\pi_tx_n)]dt+b_n(t,\pi_tx_n)d\omega(t),\quad t>0,\ t\neq t_k,\\
\label{system(4.2}
 \Delta x_n(t_k)&=x_n(t_k^+)-x_n(t_k^-)=I_k^n(x_n(t_k)),\ t=t_k,\ k=1,2,\dots,m,\\
\label{system(4.3)}
x_n(t)&=\varphi(t),\quad t\in[-r,0]\ (0\leq r<\infty),
\end{align}
where $A_n$, $f_n,$ $a_n,$ $b_n$ and $I_k^n$ satisfy the conditions of Lemma \ref{lem(3.1)}, Lemma \ref{lem3.2} and Proposition  \ref{Pro2.3} with the same constants $C_i,$ $i=1,2,\dots,5,$ $h_k$ and $h_0$ for each $n=1,2,\dots.$ Then equation (\ref{system(4.1)}) has a unique mild solution $x_n\in C([0,T],L^p(\Omega,X)).$ Hence, $x_n(t)$ satisfies the stochastic integral equation
\begin{align*}
x_n(t)&=S_n(t)[\varphi(0)+f_n(0,\varphi)]-f_n(t,\pi_tx_n)-\int_0^tA_n S_n(t-s)f_n(s,\pi_sx_n)ds\nonumber\\
&\quad+\int_0^tS_n(t-s)a_n(s,\pi_sx_n)ds+\int_0^tS_n(t-s)b_n(s,\pi_sx_n)d\omega(s)\nonumber\\
&\quad+\sum_{0<t_k<t}S_n(t-t_k)I_k^n(x_n(t_k)),\quad a.s.,\ t\in [0,T],
\end{align*}
We now present the following further assumption to consider our main result of this section.
\begin{assumption}
\label{A4}
For each $Z>0,$
$\mathop{{\rm sup}}_{|\cdot|\leq Z}|f_n(t,\cdot)-f(t,\cdot)|\rightarrow0,$ $\mathop{{\rm sup}}_{|\cdot|\leq Z}|a_n(t,\cdot)-a(t,\cdot)|\rightarrow0,$ $\mathop{{\rm sup}}_{|\cdot|\leq Z}|b_n(t,\cdot)-b(t,\cdot)|\rightarrow0$ and $\mathop{{\rm sup}}_{|\cdot|\leq Z}|I_k^n(t,\cdot)-I_k(t,\cdot)|\rightarrow0$
as $n\rightarrow\infty$ for any $t\in[0,T].$
\end{assumption}
\begin{theorem}
\label{the4.1}
Suppose that Assumptions (\ref{A1}), (\ref{A2}) and (\ref{A4}) hold. Let $x_n(t)$ and $x(t)$ be the mild solutions of equations (\ref{system(4.1)}) and (\ref{system(1.1)}), respectively. Then for each $T\in(0,\infty),$
$$\sup_{0\leq t\leq T}E|x_n(t)-x(t)|^p\rightarrow0\quad as\ n\rightarrow\infty.$$
\end{theorem}
\begin{proof}
We consider
\begin{align*}
x_n(t)-x(t)&=\psi(t)+[f_n(t,\pi_tx)-f_n(t,\pi_tx_n)]+\int_0^tA_nS_n(t-s)[f_n(s,\pi_sx)-f_n(s,\pi_sx_n)]ds\\
&\quad+\int_0^tS_n(t-s)[a_n(s,\pi_sx_n)-a_n(s,\pi_sx)]ds\\
&\quad+\int_0^tS_n(t-s)[b_n(s,\pi_sx_n)-b_n(s,\pi_sx)]d\omega(s)\\
&\quad+\sum_{k=1}^m\left\{S_n(t-t_k)[I_k^n(x_n(t_k))-I_k^n(x(t_k))]\right\},\ P-a.s.
\end{align*}
for any $t\in[0,T],$ where
\begin{align}
\label{eq4.4}
\psi(t)&=[S_n(t)-S(t)]\varphi(0)+S_n(t)[f_n(0,\varphi)-f(0,\varphi)]+[S_n(t)-S(t)]f(0,\varphi)\nonumber\\
&\quad+[f(t,\pi_tx)-f_n(t,\pi_tx)]+\int_0^tA_nS_n(t-s)[f(s,\pi_sx)-f_n(s,\pi_sx)]ds\nonumber\\
&\quad+\int_0^t[AS(t-s)-A_nS_n(t-s)]f(s,\pi_sx)ds+\int_0^tS_n(t-s)[a_n(s,\pi_sx)-a(s,\pi_sx)]ds\nonumber\\
&\quad+\int_0^t[S_n(t-s)-S(t-s)]a(s,\pi_sx)ds+\int_0^tS_n(t-s)[b_n(s,\pi_sx)-b(s,\pi_sx)]d\omega(s)\nonumber\\
&\quad+\int_0^t[S_n(t-s)-S(t-s)]b(s,\pi_sx)d\omega(s)+\sum_{k=1}^m\left\{S_n(t-t_k)[I_k^n(x(t_k))-I_k(x(t_k))]\right\}\nonumber\\
&\quad+\sum_{k=1}^m\left\{[S_n(t-t_k)-S(t-t_k)]I_k(x(t_k))\right\}.
\end{align}
By Proposition  \ref{pro2.1} (or Proposition  \ref{pro2.2} when $p=2$) and Lemma \ref{lem3.2}, we get
\begin{align*}
&\left(1-C_4\|A^{-\alpha}\|-Me^{\delta T}\sum_{k=1}^mh_k\right)^pE|x_n(t)-x(t)|^p\\
&\leq4^{p-1}\bigg\{E|\psi(t)|^p+T^{p-1}\left(\frac{M'e^{\delta T}}{t}C_4\|A^{-\alpha}\|\right)^p\int_0^tE|x_n(s)-x(s)|^pds\\
&\quad+T^{p-1}M^pe^{p\delta T}C_1\int_0^tE|x_n(s)-x(s)|^pds+c(p,T)M^pe^{p\delta T}C_2\int_0^tE|x_n(s)-x(s)|^pds\bigg\}\\
&\leq4^{p-1}E|\psi(t)|^p+L\int_0^tE|x_n(s)-x(s)|^pds,
\end{align*}
where $L=4^{p-1}e^{p\delta T}\left\{T^{p-1}\left[\frac{M'}{t}C_4\|A^{-\alpha}\|\right]^p+T^{p-1}M^pC_1+c(p,T)M^pC_2\right\}.$
Hence, by Lemma 1 from Gikhman and Skorohod \cite[p.41]{72}, we have
$$E|x_n(t)-x(t)|^p\leq4^{p-1}E|\psi(t)|^p+L\int_0^te^{L(t-s)}E|\psi(t)|^pds.$$
Therefore, to prove the theorem, it is sufficient to show that $\sup_{0\leq t\leq T}E|\psi(t)|^p\rightarrow0$ as $n\rightarrow\infty.$ First, $\sup_{0\leq t\leq T}E|[S_n(t)-S(t)]\varphi(0)|^p\rightarrow0$ and $\sup_{0\leq t\leq T}E|[S_n(t)-S(t)]f(0,\varphi)|^p\rightarrow0$ as $n\rightarrow\infty$ as shown earlier in (\ref{eq(3.5)}) and (\ref{eq(3.6)}), respectively.  To show that the remaining terms in (\ref{eq4.4}) also go to zero, we consider
\begin{align}
\label{eq4.5}
E|S_n(t)[f_n(0,\varphi)-f(0,\varphi)]|^p\leq2^{2p-1}(Me^{\delta T}C_5\|A^{-\alpha}\|)^p(1+E\|\varphi\|_c^p)<\infty,
\end{align}
\begin{align}
\label{eq4.6}
E|f_n(t,\pi_tx)-f(t,\pi_tx)]|^p\leq2^{2p-1}(C_5\|A^{-\alpha}\|)^p[1+C(T)]<\infty,
\end{align}
\begin{align}
\label{eq4.7}
E\left|\int_0^tS_n(t-s)[a_n(s,\pi_sx)-a(s,\pi_sx)]ds\right|^p&\leq(2TMe^{\delta T})^pC_3[1+C(T)]<\infty,
\end{align}
\begin{align}
\label{eq4.8}
E\left|\int_0^tS_n(t-s)[b_n(s,\pi_sx)-b(s,\pi_sx)]ds\right|^p&\leq c(p,T)(2Me^{\delta T})^pTC_3[1+C(T)]<\infty,
\end{align}
\begin{align}
\label{eq4.9}
E\left|\int_0^tA_nS_n(t-s)[f(s,\pi_sx)-f_n(s,\pi_sx)]ds\right|^p&\leq2^{2p-1}\left(T\frac{M'e^{\delta T}}{t}C_5\|A^{-\alpha}\|\right)^p[1+C(T)]<\infty
\end{align}
and
\begin{align}
\label{eq4.10}
E\left|\sum_{k=1}^mS_n(t-t_k)[I_k^n(x(t_k))-I_k(x(t_k))]\right|^p&\leq2^{2p-1}m^{p-1}M^pe^{p\delta T}\left(\sum_{k=1}^m{h_k}^p+m{h_0}^p\right)<\infty.
\end{align}
Thus, by Assumption (\ref{A4}) and Lebesgue's dominated convergence theorem, the left hand side of inequalities (\ref{eq4.5})-(\ref{eq4.10}) tend to zero as $n\rightarrow\infty$ for any $t\in[0,T].$
From Assumption (\ref{A2}), Lemma \ref{lem3.2}, Proposition \ref{Pro2.3} and dominated convergence theorem, it follows that
$$\mathop{{\rm sup}}_{0\leq t\leq T}\left|\int_0^t[AS(t-s)-A_nS_n(t-s)]f(s,\pi_sx)ds\right|^p\rightarrow0\quad as\ n\rightarrow\infty,$$
$$\mathop{{\rm sup}}_{0\leq t\leq T}\left|\int_0^t[S_n(t-s)-S(t-s)]a(s,\pi_sx)ds\right|^p\rightarrow0\quad as\ n\rightarrow\infty,$$
$$\mathop{{\rm sup}}_{0\leq t\leq T}\left|\int_0^t[S_n(t-s)-S(t-s)]b(s,\pi_sx)d\omega(s)\right|^p\rightarrow0\quad as\ n\rightarrow\infty$$
and
$$\mathop{{\rm sup}}_{0\leq t\leq T}\left|\sum_{k=1}^m\left\{[S_n(t-t_k)-S(t-t_k)]I_k(x(t_k))\right\}\right|^p\rightarrow0\quad as\ n\rightarrow\infty.$$
This completes the proof.
\end{proof}
Suppose that the coefficients in System (\ref{system(4.1)})-(\ref{system(4.3)}) depend on a parameter $\theta$ which varies through some set of numbers $G_1$:
\begin{align}
\label{system(4.11)}
 d[x_\theta(t)+f_\theta(t,\pi_tx_\theta)]&=[A_\theta x_\theta(t)+a_\theta(t,\pi_tx_\theta)]dt+b_\theta(t,\pi_tx_\theta)d\omega(t),\quad t>0,\ t\neq t_k,\\
\label{system(4.12}
 \Delta x_\theta(t_k)&=x_\theta(t_k^+)-x_\theta(t_k^-)=I_k^\theta(x_\theta(t_k)),\ t=t_k,\ k=1,2,\cdots,m,\\
\label{system(4.13)}
x_\theta(t)&=\varphi(t),\quad t\in[-r,0]\ (0\leq r<\infty),
\end{align}
where $A_\theta$ is the infinitesimal generator of a strongly continuous semigroup $\{S_\theta(t),t\geq0\}$ of bounded linear operators on $X.$ Assume that for each $N>0,$ $\mathop{{\rm sup}}_{|\cdot|\leq N}|f_\theta(t,\cdot)-f_{\theta_0}(t,\cdot)|\rightarrow0,$ $\mathop{{\rm sup}}_{|\cdot|\leq N}|a_\theta(t,\cdot)-a_{\theta_0}(t,\cdot)|\rightarrow0,$ $\mathop{{\rm sup}}_{|\cdot|\leq N}|b_\theta(t,\cdot)-b_{\theta_0}(t,\cdot)|\rightarrow0$ and $\mathop{{\rm sup}}_{|\cdot|\leq N}|I_k^\theta(t,\cdot)-I_k^{\theta_0}(t,\cdot)|\rightarrow0$ as $\theta\rightarrow\theta_0$ for any $t\in[0,T].$
\begin{corollary}
\label{cor4.2}
Let $A_\theta,$ $f_\theta,$ $a_\theta,$ $b_\theta$ and $I_k^\theta$ for each $\theta$ satisfy the conditions of Lemma \ref{lem(3.1)}, Lemma \ref{lem3.2} and Proposition \ref{Pro2.3} with the same constants. Then System (\ref{system(4.11)})-(\ref{system(4.13)}) has a unique mild solution and satisfies for each $T\in(0,\infty):$
$$\mathop{{\rm sup}}_{0\leq t\leq T}E|x_\theta(t)-x_{\theta_0}(t)|^p\rightarrow0\quad as\ \theta\rightarrow\theta_0.$$
\end{corollary}
\begin{proof}
The proof of Corollary \ref{cor4.2} is similar as the case of $\mathop{{\rm sup}}_{0\leq t\leq T}E|x_n(t)-x(t)|^p\rightarrow0\quad as\ n\rightarrow\infty$ in Theorem \ref{the4.1}.
\end{proof}

\acknowledgements{\rm We are extremely grateful to the critical comments and invaluable suggestions made by anonymous honorable reviewers. }

\end{document}